\documentclass[OLF]{amsart}
\usepackage{latexsym,amssymb}
\usepackage{newlattice}
\usepackage{graphicx}

\theoremstyle{plain}
\newtheorem{theorem}{Theorem}

\newtheorem*{algorithm}{Algorithm}
\newtheorem{lemma}{Lemma}
\newtheorem{definition}[lemma]{Definition}

\DeclareMathOperator{\Cut}{Cut}

\newcommand{\bc}{\mbf{c}}
\newcommand{\bs}{\mbf{s}}
\newcommand{\bm}{\mbf{m}}

\newcommand{\bu}{{\mbf{u}}}
\newcommand{\bv}{{\mbf{v}}}

\newcommand{\sixty}{s_{1960}}
\newcommand{\bsixty}{\bs_{\tup{\tbf{1960}}}}

\DeclareMathOperator{\At}{At}
\DeclareMathOperator{\Split}{Split}
\newcommand{\cV}{\E{V}}
\newcommand{\cC}{\E{C}}

\begin{document}

\title[On the 1960 sectional complement]{On the algorithmic construction\\of the 1960 sectional complement}

\author{G. Gr\"{a}tzer}
\address{Department of Mathematics\\
   University of Manitoba\\
   Winnipeg, MB R3T 2N2\\
   Canada}
\email[G. Gr\"atzer]{gratzer@mac.com}
\urladdr[G. Gr\"atzer]{http://server.math.umanitoba.ca/homepages/gratzer/}

\author{G. Klus}
\email[G. Klus]{garett.klus@gmail.com}

\author{A. Nguyen}
\email[A. Nguyen]{athena.nguyen@gmail.com}

\date{Sept. 11, 2009; revised: Aug. 5, 2010}
\begin{abstract}
In 1960, G. Gr\"atzer and E.\,T. Schmidt proved that every finite distributive 
lattice can be represented as the congruence lattice of a sectionally 
complemented finite lattice $L$. For $u \leq v$ in $L$, they constructed a 
sectional complement, which is now called the \emph{1960 sectional complement}. 

In 1999, G. Gr\"atzer and E.\,T. Schmidt discovered a very simple way of constructing a sectional complement in the ideal lattice of a chopped lattice made up of two sectionally complemented finite lattices overlapping in only two elements---the Atom Lemma. The question was raised whether this simple process can be generalized to an algorithm that finds the 1960 sectional complement.

In 2006, G.~Gr\"atzer and M. Roddy discovered such an algorithm---allowing a wide 
latitude how it is carried out. 

In this paper we prove that the wide latitude apparent in the algorithm is deceptive: whichever way the algorithm is carried out, it~produces the same sectional complement. This solves, in fact, Problems 2 and 3 of the Gr\"atzer-Roddy paper. Surprisingly, the unique sectional complement provided by the algorithm is the 1960 sectional complement, solving Problem 1 of the same paper.
\end{abstract}

\maketitle

\section{Introduction}\label{S: Introduction}
We assume that the reader is somewhat familiar with the field: congruences lattices of a finite lattice. For a book on this field, see G. Gr\"atzer \cite{CLFL}, for a survey paper, see G. Gr\"atzer and E.\,T. Schmidt \cite{GS04}. We are interested in the following result of G. Gr\"atzer and E.\,T. Schmidt \cite{GS62}:

\begin{theorem}
Every finite distributive lattice $D$ can be represented as the congruence lattice of a finite \emph{sectionally complemented} lattice $L$.
\end{theorem}

In \cite{CLFL}, we call the construction in \cite{GS62} the \emph{1960 construction}. In \cite{GS62} (redone also in G.~Gr\"atzer and H. Lakser \cite{GL05}), for every $u \leq v \in L$, we construct in one step a sectional complement $\bsixty$, which we call the \emph{1960 sectional complement}.

There is another result in the literature yielding sectionally complemented ideal lattices of chopped lattices:
\begin{lemma}[Atom Lemma, G. Gr\"atzer and E.\,T. Schmidt \cite{GS99a}]\label{T:atom}
Let $M$ be a chopped lattice with two maximal elements $m_1$ and $m_2$. We assume that $\id{m_1}$ and $\id{m_2}$ are sectionally complemented lattices. If $p = m_1\mm m_2$ is an atom, then $\Id M$ is sectionally complemented.
\end{lemma}

The idea of the proof is the following: we form the sectional complement componentwise. If the resulting vector is not compatible, \emph{we cut down one component} to make it compatible, and then prove that the smaller vector is compatible but still large enough to be a sectional complement. 

Based on this idea, in G. Gr\"atzer and M. Roddy \cite{GR07}, we introduce an algorithm that---starting with the local maximal sectional complements, in a finite sequence of \emph{cuts}, $\gS$---produces a sectional complement, $s_{\gS}$. We~have a wide latitude how to perform the algorithm, leading in \cite{GR07} to a series of problems, including: how many sectional complements are constructed by the algorithm, can we obtain the 1960 sectional complement with the algorithm; and so on.

In this paper, we answer all these questions with the following two results.

\begin{theorem}\label{T:firstresult}
Let $\gS$ be any sequence of cuts in the algorithm. Then the sectional complement, $\bs_\gS$, is independent of $\gS$.
\end{theorem}

In other words, the algorithm produces exactly one sectional complement, we denote it by $\bs$. This theorem solves Problems 2 and 3 of \cite{GR07}.

The second result identifies the unique sectional complement provided by Theorem~\ref{T:firstresult}:

\begin{theorem}\label{T:secondresult}
The unique sectional complement $\bs$ produced by the algorithm is the 1960 sectional complement, that is, $\bs = \bsixty$.
\end{theorem}

This solves Problems 1--3 of \cite{GR07}.

So how far does the algorithm of \cite{GR07} go? Exactly to the 1960 sectional complement.

This  paper proves that the rather difficult ad hoc original proof from 1960 can be replaced by the algorithm of Gr\"atzer and Roddy. There is still a very important problem that stays open. The algorithm produces a sectional complement under the assumption of the Atom Lemma and for the 1960 construction. Is~there a natural class of sectionally complemented finite chopped lattices, different from these two classes, in which the algorithm produces a sectional complement?

We would like to thank the referee for an unusually insightful report.

\section{Preliminaries}\label{S:Background}
 We assume that the reader is familiar with the concepts of chopped lattice, with vectors, compatible vectors, and ideals of a chopped lattice, see, for instance, the book \cite{CLFL}.

In this paper, $D$ is a finite distributive lattice and $P$ is the order of join-irreducible elements of $D$. $M$ is a chopped lattice, the 1960 construction for $D$ (for $P$). All the definitions and results that follow are about $M$.

\subsection{The chopped lattice $M$}
Figures~\ref{Fi:V}--\ref{Fi:H} are reminders of how we define the chopped lattice $M$ from the lattices $N(p,q)$, where $p \succ q \in P$, and also of the crucial orders $V$, $C$, and $H$. A cover preserving suborder $V$ will be called a $\cV$-suborder, and for $C$, we have $\cC$-suborders.

\begin{figure}[p]
\centering\includegraphics{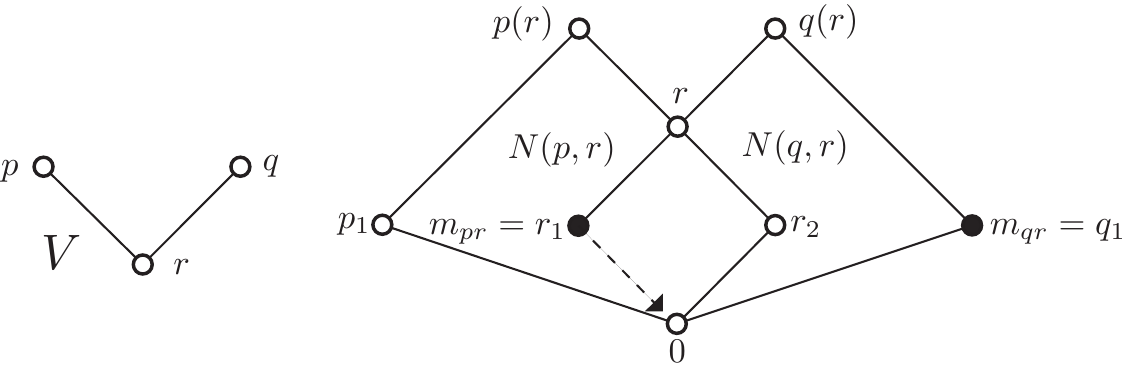}
\caption{Constructing part of $M$ for the order V (and $\cV$-cuts).}\label{Fi:V}

\bigskip

\centering\includegraphics{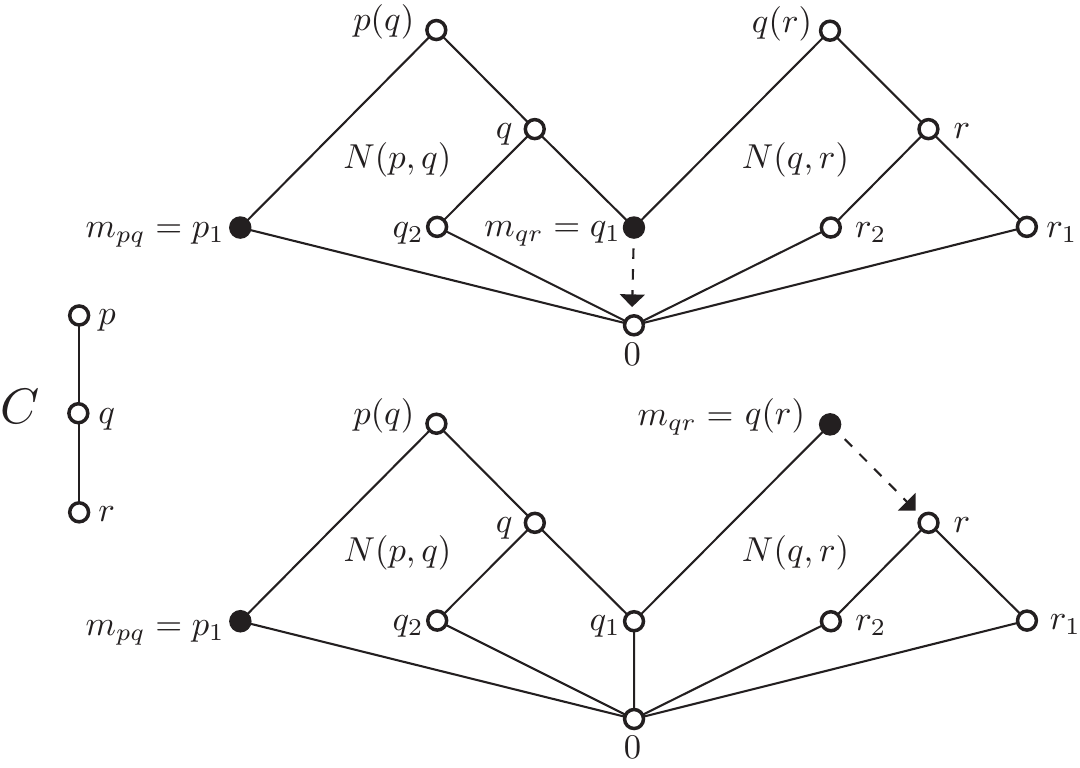}
\caption{Constructing part of $M$ for the order C  (and $\cC$-cuts).}\label{Fi:C}

\bigskip

\centering\includegraphics{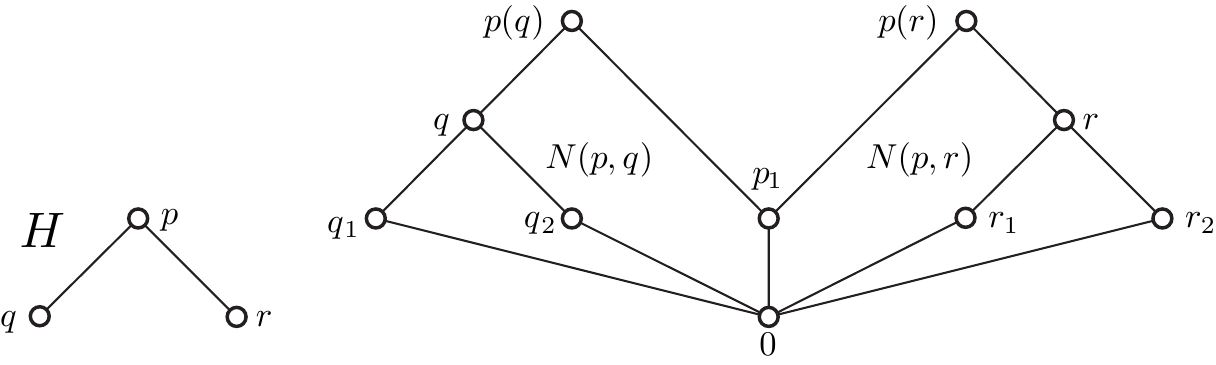}
\caption{Constructing part of $M$ for the order hat.}\label{Fi:H}
\end{figure}

We identify ideals of $M$ with vectors $\bc = (c_{xy} \mid x \succ y \in P)$.

\subsection{The basic definitions}
We need the following definitions to formulate the algorithm from \cite{GR07}.

\begin{definition}\label{D:s}
Let $\bu \leq \bv$ be compatible vectors of the chopped lattice $M$. Define the vector
$\bm = (m_{xy} \mid x \succ y \in P)$, where $m_{xy}$ is the
\emph{maximal sectional complement} of $u_{xy}$ in~$v_{xy}$.
\end{definition}

\begin{definition}\label{D:Vfail}
Let $\set{p,q,r}$ be a $\cV$-suborder.
\begin{enumeratei}
\item  A vector $\bc = (c_{xy} \mid x \succ y \in P)$ is \emph{$\cV$-compat\-i\-ble}
at $\set{p,q,r}$ \tup{(}or $\cV(p,q,r)$-\emph{compatible}\tup{)},
 if $c_{pr} \mm r = c_{qr} \mm r$ in $M$.
 Otherwise, $\bc$ is $\cV$-\emph{incompatible} at $\set{p,q,r}$ \tup{(}or $\cV(p,q,r)$-\emph{incompatible}\tup{)}.
\item The vector $\bc$ is \emph{$\cV$-compatible}, if it is $\cV$-compatible at every $\cV$-suborder in $P$.
\item The vector $\bc$ has a \emph{$\cV$-failure} at $\set{p,q,r}$
\tup{(}or $\cV(p,q,r)$-\emph{failure}\tup{)}, if $\bc$ is
$\cV$-in\-compatible at $\set{p,q,r}$  and, additionally, $c_{pr} = m_{pr}$ and $c_{qr} = m_{qr}$, that is, $\bc = \bm$ on $\set{p,q,r}$.
\end{enumeratei}
\end{definition}

\begin{definition}\label{D:Cfail}
Let $\set{p,q,r}$ be a $\cC$-suborder.
\begin{enumeratei}
\item A vector $\bc = (c_{xy} \mid x \succ y \in P)$ is \emph{$\cC$-compat\-i\-ble}
at $\set{p,q,r}$ \tup{(}or \emph{$\cC(p,q,r)$-compatible}\tup{)}, if
 $c_{pq} \mm q_1 = c_{qr} \mm q_1$  in $M$.
 Otherwise, $\bc$ is \emph{$\cC$-incompatible} at $\set{p,q,r}$ \tup{(}or \emph{$\cC(p,q,r)$-incompatible}\tup{)}.
\item The vector $\bc$ is \emph{$\cC$-compatible}, it is $\cC$-compatible at every $\cC$-suborder in $P$.
\item We say that $\bc$ has a \emph{$\cC$-failure} at $C$
\tup{(}or \emph{$\cC(p,q,r)$-failure}\tup{)}, if $\bc$ is
$\cC$-incompatible at $\set{p,q,r}$ and, additionally, $c_{pr} = m_{pr}$ and $c_{qr} = m_{qr}$, that is, $\bc = \bm$ on $\set{p,q,r}$.
\end{enumeratei}
\end{definition}

Let $\bc$ be a vector. A \emph{cut} of $\bc$ is a vector, $\Cut(\bc)$,
that agrees with $\bc$ except at one coordinate $x \succ y \in P$ (we \emph{cut at} $x, y$) and $c_{xy} \succ \Cut(\bc)_{xy}$ at $x, y$.

\begin{definition}
For a vector $\bc$ with a $\cV(p,q,r)$-failure, we define the $\cV$-cut, $\Cut_{\cV}(\bc)$, as follows. By Lemma~3 of \cite{GR07},  up to symmetry in $p$, $q$ and $r_1$, $r_2$ \lp see Figure~\ref{Fi:V}\rp,
\begin{equation}\label{Eq:Vfail}
   m_{pr} = r_1 \text{ and } m_{qr} = q_1;
\end{equation}
we define $\Cut_{\cV}(\bc)$ by cutting $\bc$ at $\set{p,r}$:
\begin{equation}\label{Eq:Vfail2}
   \Cut_{\cV}(\bc)_{pr} = 0.
\end{equation}
\end{definition}

\begin{definition}\label{D:Cfailure}
For a vector $\bc$ with a $\cC(p,q,r)$-failure, we define the $\cC$-cut, $\Cut_{\cC}(\bc)$, as follows. By Lemma~4 of \cite{GR07}, either \lp see Figure \ref{Fi:C}\rp
\begin{equation}
   m_{pq} = p_1 \text{ and } m_{qr} = q_1,
\end{equation}
and we define $\Cut_{\cV}(\bc)$ by cutting $\bc$ at $\set{q,r}$:
\begin{equation}
   \Cut_{\cV}(\bc)_{qr} = 0;
\end{equation}
or \lp see Figure \ref{Fi:C}\rp
\begin{equation}
   m_{pq} = p_1 \text{ and } m_{qr} = q(r),
\end{equation}
and we define $\Cut_{\cV}(\bc)$ by cutting $\bc$ at $\set{q,r}$:
\begin{equation}
   \Cut_{\cV}(\bc)_{qr} = r.
\end{equation}
\end{definition}

\subsection{The algorithm.}\label{S:algorithm}

We need one more definition:

\begin{definition}
A $\cC(p,q,r)$-failure for $\bc$ at $p \succ q \succ r$ is \emph{minimal}, if there is no $\cC(p',q',r')$-failure for ${\bc}$ with $q' < q$.
\end{definition}

Now we are ready to formulate the algorithm from \cite{GR07}; note that there are some minor changes to improve readability.

\begin{algorithm}
Given the compatible vectors
$\bu \leq \bv$ and the vector $\bm$ as in \tbf{Definition}~\ref{D:s}, we construct a vector $\bs$.

\tbf{Step 1.} Set ${\bc}$ = ${\bm}$.

\tbf{Step 2.} Choose an arbitrary $\cV$-failure \lp if any\rp, and perform the
corresponding $\cV$-cut, obtaining a new $\bc = \Cut_{\cV}(\bc)$. Repeat this until there are no more $\cV$-failures.

\tbf{Step 3.} Look for a \emph{minimal} $\cC$-failure \lp if any\rp, and perform
the corresponding $\cC$-cut, obtaining a new $\bc = \Cut_{\cC} (\bc)$.
Repeat this until there are no more $\cC$-failures.

Since $M$ has finitely many maximal elements and $\Cut(\bc) \leq \bc$, the process must terminate, yielding a vector $\bs$.

\end{algorithm}

The algorithm does not specify the sequence of cuts necessary to complete it. Let $\gS$ denote a particular sequence of cuts in the algorithm and let $\bs_\gS$ denote the vector obtained.

The following is the main result of G. Gr\"atzer and M. Roddy \cite{GR07}.

\begin{theorem}\label{T:IVmain}
The vector $\bs_\gS$ is compatible and it is a sectional
complement of $\bu$ in~$\bv$ in $\Id M$. Hence the lattice $\Id M$
is sectionally complemented. Moreover, for every $p \succ q$ in $P$,
either $(s_{\gS})_{pq} = m_{pq}$ or $(s_{\gS})_{pq} \prec
m_{pq}$.
\end{theorem}

\section{Proof of Theorem \ref{T:firstresult}}\label{S:Theorem1}
In this section we prove Theorem~\ref{T:firstresult}, that is, we prove that $\bs_\gS$ does not depend on the choice of $\gS$.

\subsection{The invariance for Step 2}\label{S:Cinvariance}

Let $\bm^2$ denote the following vector:
\begin{equation}\label{Eq:m2}
   m^2_{pr} =
   \begin{cases}
       0, &\text{if $m_{pr} = r_i$ and there is a $\cV(p,q,r)$-failure for some $q \succ r$;}\\
       m_{pr}, &\text{otherwise.}
   \end{cases}
\end{equation}

\begin{lemma}\label{L:Vinvariance}
At the end of Step 2, we obtain the vector $\bm^2$, independent of the sequence of $\cV$-cuts performed.
\end{lemma}

\begin{proof}
Let us look at $m_{pr}$. If there is a $\cV(p,q,r)$-failure and the corresponding $\cV$-cut was performed, then $r_i$ was replaced by $0$, see Figure~\ref{Fi:V}.

Now assume that there is a $\cV(p,q,r)$-failure but the corresponding $\cV$-cut was not performed. This can only happen if there is a $\cV(p',q',r')$-failure, the corresponding $\cV$-cut was performed, and after the cut there is no $\cV(p,q,r)$-failure. Clearly, $p = p'$, $r = r'$, and $q \neq q'$. But then the $\cV$-cut at $\set{p,q',r}$ would also replace $m_{pr}$ with $0$, so the first line of \eqref{Eq:m2} is verified.

Of course, if there is a $\cV(p,q,r)$-failure for $\bm^2$, then that would also be a $\cV(p,q,r)$-failure for $\bm$, verifying \eqref{Eq:m2}.
\end{proof}

We get something extra for the vector $\bm^2$:

\begin{lemma}\label{L:Vinvariance2}
The vector $\bm^2$ is $\cV$-compatible.
\end{lemma}

\begin{proof}
By \tbf{Lemma} \ref{L:Vinvariance}.
\end{proof}

\subsection{The invariance for Step 3}

We prove Theorem~\ref{T:firstresult} in this section. In~light of the results of Section~\ref{S:Cinvariance}, we have to prove that any sequence of $\cC$-cuts applied to $\bm^2$ as specified in Step 3 of the algorithm, yields a unique vector. This will be done in the next three lemmas.

Let $C(p,q,r)$ be a $\cC$-suborder. It will be convenient to call $r \prec q$ the \emph{stem} of $C$.

\begin{lemma}\label{L:1}
Let $\bm^2$ have a $\cC$-failure at $C(p,q,r)$.  Then any $\cC$-suborder of $P$ with the same stem, $r \prec q$, also has a $\cC$-failure.  Moreover, all these failures are resolved by the same cut.
\end{lemma}

\begin{proof}
Since $\bm^2$ has a $\cC(p,q,r)$-failure, by Lemma~4 of \cite{GR07} (restated in \tbf{Definition}~\ref{D:Cfailure}), $m^2_{pq}\mm q_1 = 0$ and $m^2_{qr}\mm q_1=q_1$. Let $C(t,q,r)$ be a $\cC$-suborder; it shares the stem with $C(p,q,r)$. By~\tbf{Lemma}~\ref{L:Vinvariance2}, the vector $\bm^2$ is $\cV$-compatible, in particular, $V(p,t,q)$ is $\cV$-compatible and so $m^2_{pq}\mm q = m^2_{tq}\mm q$. Therefore, $m^2_{pq}\mm q_1 = m^2_{tq}\mm q_1 = 0$.  Hence, $m^2_{tq}\mm q_1= 0 \neq q_1 = m^2_{qr}\mm q_1$ and so $C(t,q,r)$ is a $\cC$-failure. Since the stem of both $C(p,q,r)$ and $C(t,q,r)$ is $\set{q, r}$, the failures on $C(p,q,r)$ and $C(t,q,r)$ will be corrected (by cutting the value $m_{q,r}$) the same way.
\end{proof}

\begin{lemma}\label{C:2}
Let $C_1$ and $C_2$ be two minimal $\cC$-failures that do not share a stem. Then, after a $\cC$-cut at $C_1$, the chain $C_2$ still has a $\cC$-failure.
\end{lemma}

\begin{proof}
Since $C_1$ is a minimal $\cC$-failure, the stem of $C_1$ is not the upper covering pair of $C_2$. Since the $\cC$-cut on $C_1$ take place in the $N(p,q)$ corresponding to the stem of $C_1$, the chain $\cC$-failure $C_2$ is not effected by this cut.
\end{proof}

\begin{lemma}
Let $\gS$ be any sequence of $\cC$-cuts on $\bm^2$ such that the vector $\bm^2_{\gS}$ obtained by $\gS$ has no $\cC$-failures.  Then $\bm^2_{\gS}$ does not depend on $\gS$.
\end{lemma}

\begin{proof}
Let $C(p,q,r)$ be a $\cC$-failure. Then $m^2_{pq}=m_{pq}$ and $m^2_{qr}=m_{qr}$.  By \tbf{Lemma}~\ref{L:1}, each stem of a $\cC$-failure is cut uniquely.  By \tbf{Lemma}~\ref{C:2}, a $\cC$-failure is not effected by cutting another $\cC$-failure, unless they share a stem. Since $C(p,q,r)$ will eventually become a minimal failure, $m^2_{qr}$ will be cut uniquely by the algorithm.
\end{proof}

So we proved that $\bs_\gS$ does not depend on the choice of $\gS$. Let $\bs$ denote this vector. By Theorem~\ref{T:IVmain}, $\bs$ is a sectional complement of $\bu$ in $\bv$.

\section{Proof of Theorem \ref{T:secondresult}}

Let $\bu \leq \bv$ be vectors in $M$. Let $\bm$ be the vector defined in \tbf{Definition} \ref{D:s}. Let $\bs_{\tbf{1960}}$ denote the vector representing the 1960 sectional complement of $\bu$---see also~\eqref{Eq:atm2}.  For a vector $\bc$, let $\At(\bc)$ denote the atoms of $M$ (regarded as compatible vectors) contained in $\bc$.

Clearly, $\At(\bm)\ci \At(\bv)$, since $\bm\leq \bv$. Moreover, $\At(\bm) \ii \At(\bu) =\es$, because $\bm \mm \bu=0$. Therefore,
\begin{equation}\label{Eq:atm}
   \At(\bm)\ci \At(\bv)-\At(\bu).
\end{equation}

In this section, we need one more definition from \cite{GS62} and \cite{GL05}. Note that the index of $q_i$ is computed modulo $2$.

\begin{definition}
We say that $q \in P$ \emph{splits over} $(\bu, \bv)$, if there
exists a $p \succ q$ in $P$, with $p_1, q_{i} \in \At(\bv)-\At(\bu)$ and
$q_{i+1} \in \At(\bu)$.
\end{definition}

We denote by $\Split(\bu,\bv)$ the set of all elements $q_i$ in $\At(\bv)-\At(\bu)$ such that $q$ splits over $(\bu, \bv)$ and we borrow from \cite{GR07} and \cite{GS62} the formula:
\begin{equation}\label{Eq:atm2}
    \bs_{\tup{\tbf{1960}}} = \JJ ((\At(\bv)-\At(\bu))-\Split(\bu,\bv))
\end{equation}

\begin{lemma}\label{slarger}
The inequality $\bs_{\tup{\tbf{1960}}} \leq \bm$ holds.
\end{lemma}

\begin{proof}
Since $\bu \mm  \bsixty = 0$, it follows that $u_{pq} \mm  (\sixty)_{pq} = 0$ in $N(p, q)$, for any $p \succ q$ in $P$. Hence, $\bs_{\tup{1960}} \leq m_{pq}$, for all $p \succ q$ in $P$, therefore, $\bs_{\tup{\tbf{1960}}} \leq \bm$.
\end{proof}

\begin{lemma}\label{L:clarger}
Let $\bu \leq \bv$ be vectors in $P$. Let $\bc$ be a vector obtained in a step of the algorithm and let $\Cut(\bc)$ be the vector obtained in next step of the algorithm. If $\bs_{\tup{\tbf{1960}}} \leq \bc$, then $\bs_{\tup{\tbf{1960}}} \leq \Cut(\bc)$.
\end{lemma}

\begin{proof}
Let us assume that $\bs_{\tup{\tbf{1960}}} \leq \bc$. If the algorithm terminates at $\bc$, there is nothing to prove. If the algorithm does not terminate at $\bc$, the next step is a cut of $\bc$. We distinguish two cases.

\emph{Case 1: $\cV$-cut at $V = \set{p,q,r}$.} By symmetry, we can assume that $c_{pr} = m_{pr} = r_1$ and $c_{qr} = m_{qr} = q_1$.  Since $m_{pr} = r_1$ is the maximal sectional complement of $u_{pr}$ in $v_{pr}$, it follows that either
\begin{enumeratei}
\item $v_{pr} = r_1$ and $u_{pr} = 0$;
\item $v_{pr} = r$ and $u_{pr} = r_2$.
\end{enumeratei}
Since $v_{qr} \geq c_{qr} = q_1$, if (i) holds, then either $v_{qr} = q_1$ or $v_{qr}= q(r)$.  In both cases, then,
\begin{equation}\label{Eq:atm3}
   v_{qr} \mm r \neq v_{pr} \mm r = r_1.
\end{equation}
Hence, $v_{pr} = r$ and $u_{pr} = r_2$.  By \eqref{Eq:atm3}, since $r_1, q_1 \in \At(\bv)-\At(\bu)$ but $r_2 \in \At(\bu)$, it follows that $r$ splits over $(\bu, \bv)$.  So $\bsixty \leq \Cut_{\cV}(\bc)$, when restricted to $V$. Since $\bc$ and $\Cut_{\cV}(\bc)$ are equal outside of $V$, we conclude that $\bsixty \leq \Cut_{\cV}(\bc)$.

\emph{Case 2: $\cC$-cut at $C = \set{p,q,r}$.}
We form a $\cC$-cut at $C = \set{p,q,r}$. By \tbf{Definition}~\ref{D:Cfailure}, we have that $m_{pq}=p_1$ and $m_{qr}\geq q_1$. So $p_1,q_1 \in \At(\bm)$, and by \eqref{Eq:atm}, $p_1,q_1\in \At(\bv)-\At(\bu)$. In particular, $p_1,q_1\nin \At(\bu)$. Now  $q_1\nin \At(\bu)$ implies that $u_{pq}=q_2$ or $u_{pq}=0$. In view of $m_{pq}\jj u_{pq}=p(q)$, this yields that $u_{pq}=q_2$. Therefore, $q_2\in \At(\bu)$, and thus, $q_2\nin \At(\bv)-\At(\bu)$. Since $p_1, q_1 \in \At(\bv)-\At(\bu)$ and $q_2\in \At(\bu)$, we see that $q$ splits over $(\bu, \bv)$. Now $q_1 \nin \At(u)$  $\bs_{\tup{\tbf{1960}}} \leq \Cut_{\cV}(\bc)$ when restricted to $C$. Since $\bc$ and $\Cut_{\cC}(\bc)$ are equal outside of $C$, we conclude that $\bsixty \leq \Cut_{\cC}(\bc)$.
\end{proof}

Combining the last two lemmas, we get the inequality $\bsixty \leq \bs$. We prove the reverse inequality (completing the proof of Theorem \ref{T:secondresult}) in the following statement.

\begin{lemma}
Let $\bc$ be a compatible vector for which $\bs_{\tup{\tbf{1960}}} \leq \bc \leq \bm$. Then $\bc = \bs_{\tup{\tbf{1960}}}$.
\end{lemma}

\begin{proof}
Let us assume that $\bc = \bsixty$ fails, that is, $\bsixty < \bc$. Then $(\sixty)_{qr} < c_{qr}$, for some $q \succ r$ in $P$. So $c_{qr}>0$. We consider three cases:

\emph{Case 1: $c_{qr}= q_1$.} Then $(\sixty)_{qr}=0$. Clearly, since $\bc\leq \bm$, we have that $\At(\bc) \ci \At(\bm)$. Hence, $q_1\in \At(\bm)$, and by \eqref{Eq:atm}, $q_1 \in \At(\bv)-\At(\bu)$. However, $q_1 \nin \At(\bsixty)$, so by \eqref{Eq:atm2}, $q_1\in \Split(\bu,\bv)$. Therefore, $p_1\in \At(\bv)-\At(\bu)$, for some $p\succ q$, and $q_2\nin \At(\bv)-\At(\bu)$.

Since $p_1,q_1\in \At(\bv)$, clearly, $v_{pq}=p(q)$. Then $q_2\in \At(\bv)$ and $q_2 \nin \At(\bv)-\At(\bu)$, so $u_{pq}=q_2$. Therefore, $m_{pq}=p_1$. Since $c_{pq}\leq m_{pq}=p_1$, it follows that $c_{pq}\mm q_1=0$. This contradicts that $\bc$ is compatible, indeed, $c_{pq}\mm q_1=0$ and $c_{qr} \mm q_1=q_1$ (since $c_{qr}= q_1$).

\emph{Case 2: $c_{qr}=r_i$, for $i=1$ or $2$.} Observe that if
$m_{qr}\geq r$, then $r_1,r_2\in \At(\bm)$, and it follows from
\eqref{Eq:atm} that $r_1,r_2 \in \At(\bv)-\At(\bu)$. Hence,
$r_1,r_2 \in \At(\bsixty)$ by \eqref{Eq:atm2}. Therefore, $(s_{\tup{{1960}}})_{qr}\geq r$, which contradicts the assumption that $c_{qr}> (s_{\tup{{1960}}})_{qr}$. So
\[
   m_{qr}=c_{qr}=r_i.
\]
Then
\begin{align}
   c_{qr}\mm r&= r_i,\label{Eq:needed}\\
   (\sixty)_{qr}&=0.
\end{align}
Hence, $r_i\in \At(\bm)$ and thus, by \eqref{Eq:atm}, $r_i \in \At(\bv)-\At(\bu)$. However, $r_i\nin \At(\bsixty)$, so $r$ splits over $(\bu,\bv)$ by \eqref{Eq:atm2}. But since $q_1 \not\in \At(\bv)-\At(\bu)$, there exists a covering pair $p \succ r$ in $P$ such that $p_1 \in \At(\bv)-\At(\bu)$ but $r_{i+1}\nin \At(\bv)-\At(\bu)$. Since $p_1,r_i\in \At(\bv)$, we conclude that $v_{pr}=p(r)$. Then $r_{i+1}\in \At(\bv)$ and $r_{i+1}\nin \At(\bv)-\At(\bu)$, so $u_{pr}=r_{i+1}$. We conclude that $m_{pr}=p_1$. Therefore, $c_{pr}\leq m_{pr}=p_1$, implying that  $c_{pr}\mm r=0$, contradicting that $\bc$ is compatible and $c_{qr}\mm r= r_i$ by \eqref{Eq:needed}.

\emph{Case 3: $c_{qr}\geq r$.} Since $r_1,r_2 \in \At(\bm)$, it follows from \eqref{Eq:atm} that $r_1,r_2 \in \At(\bv)-\At(\bu)$. So $r_1,r_2\nin \Split(\bu,\bv)$. Hence $r_1,r_2 \in \At(\bsixty)$, which implies that $(s_{1960})_{qr}\geq r$.

Since
$m_{qr}\geq c_{qr} > (s_{1960})_{qr}$, there is only one possibility:
\begin{align}
   c_{qr}=m_{qr}&=q(r),\label{Eq:alsoneeded}\\
   (s_{1960})_{qr}&=r.
\end{align}
So $q_1\in \At(\bv)-\At(\bu)$ and $q_1\nin \At(\bsixty)$, so we conclude that
$q_1\in \Split(\bu,\bv)$. Then, for some $p\succ q$,
\begin{align*}
   p_1&\in \At(\bv)-\At(\bu),\\
   q_2&\nin \At(\bv)-\At(\bu).
\end{align*}
Since $p_1,q_1\in \At(\bv)$, we conclude that $v_{pq}=p(q)$. Then $q_2\in \At(\bv)$, but $q_2\nin \At(\bv)-\At(\bu)$, so $u_{pq}=q_2$ and $m_{pq}=p_1$. Therefore, $c_{pq}\leq m_{pq}=p_1$, implying that $c_{pq}\mm q_1 =0$, contradicting that $\bc$ is compatible and $c_{qr}\mm q_1=q_1$ by \eqref{Eq:alsoneeded}.

Since each case leads to a contradiction and the three cases cover all possibilities, we conclude that $\bc =
\bsixty$.
\end{proof}

\end{document}